\documentclass{article}
\usepackage{amsmath,amsfonts}
\DeclareMathSymbol{\blacksquare}  {\mathord}{AMSa}{"04}

\newcommand{\cl}{{\mathcal Cl}}

\newtheorem{theorem}{Theorem}[section]

\newtheorem{lemma}{Lemma}[section]
\newtheorem{example}{Example}[section]

\newtheorem{definition}{Definition}[section]

\newcommand{\BC}{{\rm \kern.24em \vrule width.05em height1.4ex depth-.05ex
            \kern-.26em C}}
\newcommand{\BR}{{\rm\hskip 0.1pt I\hskip -2.15pt R}}

\renewcommand{\a}{ {\bf a}}
\newcommand{\vv}{{\bf v}}
\newcommand{\e}{{\bf e}}

\def\BC{\mathbb{C}}

\def\BR{\mathbb{R}}

\def\BN{\mathbb{N}}
\newcommand{\proof}{{\bf Proof: }}
\date{}

\makeatletter
\def\proof{\@ifnextchar[{\@proofa}{\@proofb}}
\def\@proofb{\textbf{Proof:} }
\def\@proofa[#1]{\textbf{Proof #1:} }
\makeatother

\begin{document}

\title{Almansi Theorems in Umbral Clifford Analysis and the Quantum Harmonic Oscillator}
\author
 {
 Guangbin Ren
\thanks {
Research supported by the \textit{Unidade de Investiga\c c\~ao
``Matem\'atica e Aplica\c c\~oes''} of University of Aveiro, and by
the NNSF  of China
 (No. 10771201), NCET-05-0539.
}
\\
   University of Science and Technology of China,
 Department
\\ of Mathematics, Hefei, Anhui 230026, P. R. China
\\
current address: Universidade de Aveiro, Departamento de
Matem\'atica
\\ P-3810-193 Aveiro, Portugal
\\E-mail:  ren@ua.pt
\and Nelson Faustino\thanks{Research supported by {\it Funda\c c\~ao para a Ci\^encia e a
Tecnologia} under the grant SFRH/BD/17657/2004}
\\ Universidade de Aveiro, Departamento de
Matem\'atica
\\ P-3810-193 Aveiro, Portugal
\\E-mail:  nfaust@ua.pt
}

 \maketitle

\newpage
\begin{center}
{\Large  \textbf{Almansi Theorems in Umbral Clifford Analysis and
the Quantum Harmonic Oscillator}}

\vskip 1.0truecm

\textbf{MOS subject classification:} 30$\,$G$\,$35, 35$\,$C$\,$10,
39$\,$A$\,$12, 70$\,$H$\,$05

\textbf{PACS 2006:} 02.30.$F_n$, 03.65.-$w$
\end{center}

\vskip 2.0truecm

\begin{abstract}
We introduce the Umbral calculus into Clifford analysis starting
from  the abstract of the Heisenberg commutation relation
$[\frac{d}{dx}, x] = {\bf id}$. The Umbral Clifford analysis
provides an effective framework in continuity and discreteness.

In this paper we consider functions defined in a star-like domain
$\Omega \subset \BR^n$ with values in the Umbral Clifford algebra
$C\ell_{0,n}'$ which are Umbral polymonogenic with respect to the
(left) Umbral Dirac operator $D'$, i.e. they belong to the kernel of
$(D')^k$. We prove that any polymonogenic function $f$ has a
decomposition of the form
$$f=f_1+ x'f_2 + \cdots + (x')^{k-1}f_k,$$
where $x'=x'_1e_1 + \cdots + x'_ne_n$ and $f_j,\;j=1,\ldots ,k,$ are
Umbral monogenic functions. As examples, this result recoveries  the
continuous version of the classical Almansi theorem for derivatives
and establishes the discrete version of Almansi theorem for
difference operator. The approach  also provides a similar result in
quantum field about Almansi decomposition related to Hamilton
operators.

Some concrete examples will presented for the discrete analog version of Almansi Decomposition
and for the quantum harmonic oscillator.

\end{abstract}


\newpage

\section{Introduction}

We introduce the Umbral Calculus into the Clifford analysis. This
provides  a framework unifying continuity and discreteness. We
verifies this phenomena by establishing the Almansi decomposition in
Umbral Clifford analysis.

The essence behind continuity and  discreteness is Umbral Calculus.
Umbral Calculus has many applications in combinatorics, special
function theory, approximation theory, statistics, probability,
topology and physics (see \cite{BLR98}).  Umbral calculus
(\cite{roman84, RR}) can be viewed as an abstract theory of the
Heisenberg commutation relation $[P, M] = {\bf id}$. In quantum
mechanics $P$ is the derivative and $M$ the coordinate operator. The
unification of the Umbral structure can be seen  by taking the
derivative $P$ to be the usual derivative and the divided
difference.

 In the Umbral Clifford analysis, we take the delta operator $O_{x_j}$ as momentum operators
and take
 $x_j'=\frac{1}{2}\left(x_j{O'}_{x_j}^{-1}+{O'}_{x_j}^{-1}x_j\right)$ as position operators, where
 ${O'}_{x_j}f(\underline{x})=O_{x_j}(x_jf(\underline{x}))-x_jO_{x_j}f(\underline{x})$ is the Pincherle derivative.
They  satisfy the Heisenberg-Weyl relations
$$[O_{x_j}, O_{x_k}]=0=[x_j', x_k'], \qquad [O_{x_j},
x_k']=\delta_{jk}{\bf id}.$$

The Umbral Clifford analysis is a function theory for null solutions for
the Umbral Dirac operator
$$D':=\sum_{j=1}^n \e_j O_{x_j}.$$

where $\e_1,\ldots, \e_n$ are the Clifford Algebra generators subject to the anti-commuting
relations
$$\e_j\e_k+\e_k\e_j=-2\delta_{jk}.$$

By taking  $O_{x_j}$ to be the partial derivative
$\frac{\partial}{\partial x_j}$, we recovery the standard Dirac operator $D=\sum_{j=1}^n \e_j \partial_{x_j}$. The Clifford
analysis is a refinement of harmonic analysis and a generalization
of complex analysis to higher dimensions (see \cite{BDS,DSS,GM,
GS}). In the spherical harmonics, the Fischer decomposition of
polynomials plays a central role. The Almansi decomposition
generalized the Fischer decomposition from polynomials to
polyharmonic functions. The polyhamonic theory has many applications
 in
elasticity \cite{LV},
  radar imaging \cite{AEG},
and approximation  \cite{AGH, K}.

\medskip

\textbf{Almansi's Theorem} (cf. \cite{ALM,ACL}) {\em {If $f$ is
polyharmonic of degree $k$ in a starlike domain with center $0$,
then there exist uniquely defined functions $f_1, \cdots,f_{k}$,
each harmonic in $\Omega$ such that
$$f(x)=f_1(x)+|x|^2 f_2(x)+\cdots+ |x|^{2(k-1)}f_{k}(x).$$}}

One can find important applications and generalizations of this
result in the case of several complex variables in the monograph of
\textsc{Aronszajn, Creese and Lipkin,} \cite{ACL}, e.g. concerning
functions holomorphic in the neighborhood of the origin in $\BC^n$.
More recent generalizations of Almansi's Theorem can be found in
\cite{CCGS, R_AD, R_U, MAGU}.

In $\cite{MR02}$, Malonek and Ren obtained the Almansi decomposition
in $\cl_{0,n}$. In this paper, we establish  the Almansi
decomposition  in the setting of  Umbral Clifford analysis. As a
corollary obtained a discrete version  of Almansi decomposition. The
approach  also provides a similar result in quantum field about
Almansi decomposition related to Hamilton operators.

\section{Umbral calculus revised}

\subsection{Shift-Invariant Operators}

Let $\underline{x}=(x_1,x_2,\ldots,x_n)\in \BR^n$. A monomial over
$\underline{x}$ is the product
$\underline{x}^{\alpha}=x_1^{\alpha_1}x_2^{\alpha_2}\ldots
x_n^{\alpha_n}$. Here and elsewhere $\alpha,\beta,\ldots$ denote
indices of nonnegative integers. A {\it polynomial} is a finite
linear combination of monomials. We denote the ring of polynomials
over $\underline{x}$ by $\BR[\underline{x}]$,
$\partial_{x_j}:=\frac{\partial}{\partial_{x_j}}$ the partial
derivative with respect to $x_j$ and
$\partial_{\underline{x}}:=(\partial_{x_1},\partial_{x_2},\ldots,\partial_{x_n})$
the gradient operator.

In the sequel, it will be helpful to introduce the shorthand
notations
\begin{eqnarray*}
 \partial^{\alpha}_{\underline{x}}:=\partial_{x_1}^{\alpha_1}\partial_{x_2}^{\alpha_2}\ldots\partial_{x_n}^{\alpha_n}, \\ \alpha!=\alpha_1! \alpha_2 !\ldots \alpha_n! , \\
 \left(\begin{array}{cc} \beta \\ \alpha \end{array}\right)=\frac{\beta!}{\alpha ! (\beta-\alpha)!}, \\
 |\alpha|=\sum_{j=1}^n\alpha_j.
\end{eqnarray*}

Set $\underline{y}\cdot \partial_{\underline{x}}=\sum_{j=1}^n y_j
\partial_{x_j}$ and define
the shift operator
$$\tau_{\underline{y}}f(x)=\exp(\underline{y} \cdot
\partial_{\underline{x}})f(\underline{x})=f(\underline{x}+\underline{y}).$$

 An operator $Q$ is shift-invariant if it commutes with the shift
operator $\tau_{\underline{y}}$, when acting on
$\BR[\underline{x}]$, i.e.
\begin{eqnarray}
\label{shiftinv}[Q,\tau_{\underline{y}}]p(\underline{x})=0
\end{eqnarray}
 holds for all $p \in \BR[\underline{x}]$ and $\underline{y}\in
 \BR^n$, where $[{\bf a},{\bf b}]={\bf a}{\bf b}-{\bf b}{\bf a}$
 denotes the commutator bracket between ${\bf a}$ and ${\bf b}$.
In the language of co-algebra, the shift operator
$\tau_{\underline{y}}$ is nothing else than the co-product
 in a Hopf algebra over $\BR[\underline{x}]$, but we shall
not discuss this here, see e.g. \cite{rota79} (Section V.3) for more
details.

Furthermore the following multivariate expansion formula holds
\cite{BL96}:

\begin{theorem} \label{expThm}
A linear operator $Q: \BR[\underline{x}]\rightarrow \BR[\underline{x}]$ is shift-invariant if and only if it can be expressed (as a convergent series) in the gradient $\partial_{\underline{x}}$, that is
$$Q=\sum_{|\alpha|=0}^{\infty}a_\alpha \partial_{\underline{x}}^{\alpha}$$ where $a_\alpha=\left[\frac{Q \underline{x}^\alpha}{\alpha!}\right]_{\underline{x}=\underline{0}}$.
\end{theorem}

We will denote by $\mathcal{U}$ the set of all shift-invariant
operators over $\BR[\underline{x}]$. This set, endowed with the
standard addition and composition between functions, forms an
$\BR-$algebra. In the language of umbral calculus, $\mathcal{U}$ is
usually named as the Umbral Algebra.

The Pincherle derivative of an operator $Q$ at the coordinate $x_j$ is defined as the commutator
\begin{eqnarray}
\label{Pincherle} Q'_{x_j}=[Q,x_j],
\end{eqnarray}
where on the right hand side of (\ref{Pincherle}), $x_j$ acts as a multiplication operator on $\BR[\underline{x}]$, i.e.
 $x_j: p(\underline{x})\mapsto x_jp(\underline{x})$.

According to Theorem \ref{expThm}, shift-invariant operators over
polynomials are expressed in terms of derivatives. In particular,
the multi-index derivatives $\partial_{\underline{x}}^{\alpha}$ and
linear combinations of there are shift-invariant.

Regardless the last viewpoint, one looks to shift-invariant
operators as convergent power series
$Q(\underline{x})=\sum_{|\alpha|=0}^{\infty}a_\alpha
{\underline{x}}^{\alpha}$, by replacing the vector variable
$\underline{x}$ by the gradient $\partial_{\underline{x}}$, namely
$\iota [Q(\underline{x})]=Q(\partial_{\underline{x}})$, where
$\iota:\BR[\underline{x}] \rightarrow \mathcal{U}$ is a mapping
between the algebra of polynomials and the umbral algebra. This
mapping is one-to-one and onto and therefore an algebra isomorphism.

The inverse image under $\iota:\BR[\underline{x}] \rightarrow
\mathcal{U}$ of a shift-invariant operator is known as its
indicator. As a further consequence,
$\partial_{x_j}Q(\underline{x})$ is the indicator of the Pincherle
derivative $Q'_{x_j}:=Q'_{x_j}(\partial_{\underline{x}})$ c.f.
\cite{roman84}. Therefore, the Pincherle derivative is a derivation
operation in the umbral algebra $\mathcal{U}$.

\subsection{Quantum Mechanics representation of Umbral Calculus} \label{quantumM}

The main purpose of this subsection is to gather a fully rigorous description of umbral calculus in terms of Boson calculus
associated to the second quantization. First we recall some basic facts on Quantum Field Theory.

The Heisenberg-Weyl algebra $\mathcal{F}$ is the free algebra
generated from the vacuum vector $\Phi$ by the $2n+1$ elements
$\a_1^-,\ldots,\a_n^-,\a_1^+,\ldots,\a_n^+,{\bf id}$ satisfying
$\a_j\Phi=0$, $j=1,\ldots,n$ and the commutation relations:
\begin{eqnarray}
\label{weylh}[\a_j^-,\a_k^-]=0=[\a_j^+,\a_k^+], & [\a_j^-,\a_k^+]=\delta_{jk}{\bf id}.
\end{eqnarray}

We now equip $\mathcal{F}$ with an euclidian inner product $\langle \cdot | \cdot \rangle$
such that $\langle \Phi|\Phi \rangle=1$ and the operators
$\a_j^+$ are adjoint to $\a_j^-$:
$$\langle \a_j^+ x|y \rangle=\langle x|\a_j^- y\rangle.$$

The later vector space $(\mathcal{F},\langle \cdot | \cdot \rangle)$ becomes the Bose algebra, where the
elements $\a_j^+$ and $\a_j^-$ (creation and annihilation operators, respectively)
forms a basis for the $n-$particle space.

The explicit expression from the basic vectors in the Fock spaces $\eta_\alpha \in \mathcal{F}$ can be derived from the
standard lemma of elementary Quantum Field Theory as the multi-index product of creation operators acting on the vacuum vector $\Phi$, i.e.
\begin{eqnarray}
\label{quantumV}\eta_\alpha=\prod_{j=1}^n(\a_j^+)^{\alpha_j}\Phi
\end{eqnarray}
and moreover, inner products between basic vector fields has the form $$\langle \eta_\alpha|\eta_\beta
\rangle =\alpha ! \delta_{\alpha,\beta}$$

In the following, we will describe the algebra of multivariate polynomials $\BR[\underline{x}]$ as a realization of the Bose algebra $\mathcal{F}$.

Among the infinity of possible representations of  the simplest one are the multiplication $x_j$ and the derivation operator $\partial_{x_j}$.

Now consider the multivariate monomials $\underline{x}^{\alpha} \in \BR[\underline{x}]$.
Evidently $x_j$ and $\partial_{x_j}$ acting on $\underline{x}^{\alpha}$ gives
\begin{eqnarray*}
x_j \underline{x}^{\alpha}&=&\underline{x}^{\alpha+\vv_j}, \\ \partial_{x_j}
\underline{x}^{\alpha}&=&\alpha_j\underline{x}^{\alpha-\vv_j}
\end{eqnarray*}

We now construct the umbral calculus through representations of
commutation relations (\ref{weylh})
 by position and momentum operators, $x_j'$ and $O_{x_j}$, respectively such that the action of
$x_j',O_{x_j}$ on certain polynomials $V_\alpha(\underline{x})$ is
analogous to the action of $x_j$ and $\partial_{x_j}$ on the monomials. More
specifically we shall search for $x_j'$, $O_{x_j}$ and their
associated polynomials $V_\alpha(\underline{x})$ that
satisfy
\begin{eqnarray}
\label{xQ}
x_j' V_{\alpha}(\underline{x})=V_{\alpha+\vv_j}(\underline{x}), & O_{x_j}V_{\alpha}(\underline{x})=\alpha_jV_{\alpha-\vv_j}(\underline{x}).
\end{eqnarray}

These operators can be immediately recognized as raising and
lowering operators acting on the polynomials $V_\alpha(\underline{x})$.

In the sequel, we assume that $O_{x_j}$ satisfy the shift-invariance property (\ref{shiftinv}) and $O_{x_j} (x_j)$ is a nonzero constant.

According to the umbral setting such an operator $O_{x_j}$ is called
a delta operator and such a polynomial sequence $\{
V_\alpha~:~\alpha \in \BN_0^n \}$ generated by (\ref{xQ}) is called
a {\it Sheffer set} for the multivariate delta operator
$O_{\underline{x}}=(O_{x_1},O_{x_2}, \ldots ,O_{x_n})$; the {\it
Sheffer set} $\{ V_\alpha~:~\alpha \in \BN_0^n \}$ is called {\it
basic} if $V_{\alpha}({\bf 0})=0$ and $V_{{\bf
0}}(\underline{x})={\bf 1}$. As a further consequence
$(O_{x_j}')^{-1}$ always exists and commute with $O_{x_j}$
(c.f.~\cite{roman84},~pp.~18).

Using the quantum field representation (\ref{quantumV}), the {\it basic} polynomials $V_\alpha(\underline{x})$ are obtained through the action of
$(\underline{x}')^{\alpha}:=\prod_{k=1}^n(x_k')^{\alpha_k}$
\begin{eqnarray}
\label{rodrigues}V_\alpha(\underline{x})=(\underline{x}')^{\alpha}{\bf 1},
\end{eqnarray}
which is known as the {\it Rodrigues formula}.

A further consequence of (\ref{xQ}) is the eigenproperty
\begin{eqnarray}
\label{eigenP}\sum_{j=1}^n x_j' O_{x_j} V_\alpha(\underline{x})=|\alpha|V_\alpha(\underline{x})
\end{eqnarray}
This fact provides a good motivation for calling
\begin{eqnarray}\label{umbralE}E'=\sum_{j=1}^n x_j' O_{x_j}\end{eqnarray}the umbral Euler operator,
i.e. an operator who measures the degree of $V_\alpha$;

The properties of Sheffer-type polynomials are naturally handled
within the so called {\it umbral calculus}. Incidentally we shall
study the Sheffer operator
\begin{eqnarray}\label{eq:1.41}
\Psi_{\underline{x}}: \underline{x}^\alpha \mapsto
V_\alpha(\underline{x}),
\end{eqnarray}
 where
$V_\alpha(\underline{x})$ is a Sheffer sequence. This linear
operator is obviously invertible and its inverse
$\Psi_{\underline{x}}^{-1}$ corresponds to
$\Psi_{\underline{x}}^{-1}: V_\alpha(\underline{x}) \mapsto
\underline{x}^\alpha$.

This correspondence naturally induces a correspondence between the operators $x_j,\partial_{x_j}$ and $O_{x_j},x_j'$. Indeed, according to (\ref{xQ})
the linear extension of
$\Psi_{\underline{x}}: \underline{x}^\alpha \mapsto V_\alpha(\underline{x})$ is an intertwining operator i.e.
\begin{eqnarray}
\label{intertwining}O_{x_j}\Psi_{\underline{x}} =\Psi_{\underline{x}} \partial_{x_j}, & x_j'\Psi_{\underline{x}} =\Psi_{\underline{x}} x_j.
\end{eqnarray}
and hence it intertwines the umbral Euler operator (\ref{umbralE}) and the classical Euler operator $E_{\underline{x}}=\sum_{j=1}^n x_j \partial_{x_j}$, i.e.
\begin{eqnarray}
\label{intertwiningE} E' \Psi_{\underline{x}}=\Psi_{\underline{x}}E.
\end{eqnarray}

Another important consequence of (\ref{intertwining}) is that $\Psi_{\underline{x}}$ preserves the Heisenberg-Weyl algebra relations (\ref{weylh}).

A formal power series is an infinite linear combination of
monomials. We will use the variables $\underline{x}$ and
$\underline{t}$ to denote formal power series whose coefficients are
polynomials and $$\underline{t} \cdot \underline{x}'=\sum_{j=1}^n
t_j x_j'$$ to denote the formal Euclidean inner product between
$\underline{t}$ and $\underline{x}'$.

An important consequence of the Heisenberg-Weyl representation (\ref{weylh}) concerns the exponential generating function
for Sheffer sequences $V_\alpha(\underline{x})$, which is defined by

\begin{eqnarray}\label{umbralExp}V(\underline{x},\underline{t})=\sum_{|\alpha|=0}^\infty
V_\alpha(\underline{x})\frac{\underline{t}^{\alpha}}{\alpha
!}=\sum_{k=0}^\infty \frac{1}{k!}(\underline{t} \cdot
\underline{x}')^k{\bf 1}=exp({\underline{t}\cdot
{\underline{x}'}}){\bf 1}.
\end{eqnarray}
The above formula corresponds to a formal power series in the variables $\underline{x}$ and
$\underline{t}$. According to (\ref{intertwining}) the exponential generating function (\ref{umbralExp})
may be characterized by means of $V(\underline{x},\underline{t})=\Psi_{\underline{x}}(\exp(\underline{t} \cdot \underline{x}))$.

Let us take a close look for the characterization of Sheffer sequences on the umbral algebra.

According to the Theorem \ref{expThm}, $\{ V_\alpha~:~\alpha \in \BN_0^n \}$ is a Sheffer sequence if there is a
set of multivariate polynomials $\{ O_j~:~j=1,\ldots,n\}$ such that
\begin{eqnarray}
\label{lowering}O_j(\partial_{\underline{x}})V_{\alpha}(\underline{x})=\alpha_jV_{\alpha-\vv_j}(\underline{x})
\end{eqnarray}
which are lowering operators. This characterization is not unique, i.e., there are a lot of Sheffer-type sequences satisfying equation (\ref{lowering}), for a
given $O_j(\underline{x})$. We can further classify them by postulating the existence of the associated raising operator. A general theorem
(c.f.~\cite{rota73,roman84})
states that a polynomial sequence $\{ V_\alpha~:~\alpha \in \BN^n_0\}$ satisfying the monomiality principle (\ref{xQ})
with a multi-variate operator $O_{\underline{x}}=(O_{x_1},O_{x_2},\ldots,O_{x_n})$  is uniquely determined by the polynomial vector-field
$O(\underline{x})=(O_{1}(\underline{x}),O_{2}(\underline{x}),\ldots,O_{n}(\underline{x}))$ such that
\begin{eqnarray}
O_j(\underline{0})=\underline{0}~,
&~~[\partial_{x_j}O_j(\underline{x})]_{\underline{x}=\underline{0}}
\neq 0
&\mbox{and}~~[\partial_{x_j}O_k(\underline{x})]_{\underline{x}=\underline{0}}
= 0~\mbox{for}~j \neq k.
\end{eqnarray}

The exponential generating function for $V_\alpha(\underline{x})$ is then equal to
\begin{eqnarray}
V(\underline{x},\underline{t})=\Psi_{\underline{x}}(\exp(\underline{t} \cdot \underline{x}))=\exp(\underline{x}\cdot {O}^{-1}(\underline{t}))
\end{eqnarray}
where ${O}^{-1}(\underline{t})
:=({O}^{-1}_1(\underline{t}),{O}^{-1}_2(\underline{t}),\ldots,{O}^{-1}_n(\underline{t}))$.

Furthermore their associated raising and lowering operators
associated to (\ref{xQ}) are given by
\begin{eqnarray}
\label{Oxj}O_{x_j}=O_j(\partial_{\underline{x}}), \\
\label{xj'} x_j'=x_j (O_{x_j}')^{-1}.
\end{eqnarray}

We would like to emphasize that there many possibilities still for constructing operators $x_j'$ having the raising property. In particular, it is also interesting to consider the special choice
\begin{eqnarray}
\label{xj'2} x_j'=\frac{1}{2}(x_j (O_{x_j}')^{-1}+(O_{x_j}')^{-1}x_j)
\end{eqnarray}

Since $(O_{x_j}')^{-1}$ commutes with $O_{x_j}$ (c.f.~\cite{roman84},~pp.~18), one can easily verify the Heisenberg-Weyl commutation relations (\ref{weylh}) for the operators $O_{x_j}$ and $x_j'$ defined in (\ref{Oxj}) and (\ref{xj'2}).

One of major reasons to consider (\ref{xj'2}) over (\ref{xj'}) is that under some special choices of $O_{x_j}$, the operators $O_{x_j}$ and $x_j'$ becomes symmetric and self-adjoint on a Hilbert space c.f. \cite{DHS96}.

Now we would like to point out some properties concerning the umbral operators (\ref{Oxj}) and (\ref{xj'}):

Since $O_{x_j}$ is a delta operator, the operators
$(O_{x_j}')^{-1}$ always exists and $O_{x_j}'$, $(O_{x_j}')^{-1}$ commute with $O_{x_j}$
(c.f.~\cite{roman84},~pp.~18), i.e.
\begin{eqnarray}
\label{weylh2}[O_{x_j}',O_{x_k}]=0=[(O_{x_j}')^{-1},O_{x_k}]
\end{eqnarray}
However the same in general does not hold between
$O_{x_j}'$, $(O_{x_j}')^{-1}$ and $x_j'$ defined {\it viz} formula (\ref{xj'}) (see Example~\ref{example1}).

We would like also to emphasize that there many possibilities still for constructing operators $x_j'$ having the raising property. In particular, it is also interesting to consider the special choice
\begin{eqnarray}
\label{xj'2} x_j'=\frac{1}{2}(x_j (O_{x_j}')^{-1}+(O_{x_j}')^{-1}x_j)
\end{eqnarray}
One can easily verify the Heisenberg-Weyl commutation relations (\ref{weylh}) for the operators $O_{x_j}$ and $x_j'$ defined in (\ref{Oxj}) and (\ref{xj'2}). On the other hand we also get the raising property for the operators (\ref{xj'2}). However like for (\ref{xj'}), the operators (\ref{xj'2}) does not commute with $O_{x_j}'$ and $(O_{x_j}')^{-1}$.

In brief, one of major reasons to consider (\ref{xj'2}) instead of (\ref{xj'}) is that under some special choices of $O_{x_j}$, the operators $O_{x_j}$ and $x_j'$ becomes symmetric and self-adjoint on a Hilbert space c.f. \cite{DHS96}.

Here there are some examples of so obtained representations of the Heisenberg-Weyl algebra:

\begin{example}[Forward/Backward differences]\label{example1}

Denote by $\partial_h^{\pm j}$ the forward/backward difference operators by
$$ \partial_{h}^{\pm j}=\pm \frac{\tau_{\pm h \vv_j}-{\bf id}}{h}=\pm \frac{1}{h}\left(\exp(\pm h\partial_{x_j})-{\bf id}\right).$$

These operators mutually commute when acting on functions, i.e.
\begin{eqnarray}
\label{annihilation}\partial_h^{\pm j}(\partial_h^{\pm k}f(\underline{x}))=\partial_h^{\pm k}(\partial_h^{\pm j}f(\underline{x})). \end{eqnarray}

and are interrelated with shifts $\tau_{\pm h \vv_j}$
\begin{eqnarray*}
\tau_{- h \vv_j}(\partial_h^{+j}f)(\underline{x})=(\partial_h^{-j}f)(\underline{x}), & \tau_{h \vv_j}(\partial_h^{-j}f)(\underline{x})=(\partial_h^{+j}f)(\underline{x}).
\end{eqnarray*}
Let us remark that the finite difference action $\partial_h^{\pm j}$
acting on functions satisfies the product rule
\begin{eqnarray}
\label{duality}\partial^{\pm j}_{h}(x_j f(\underline{x}))=\pm
\displaystyle\frac{(x_j\pm
h)f(\underline{x}+h\vv_j)-x_jf(\underline{x})}{h} = x_j
\partial^{\pm j}_{h}(x_j f(\underline{x}))+f(\underline{x}+h\vv_j).
\end{eqnarray}
and hence $(\partial^{\pm j}_{h})'=\tau_{\pm h\vv_j}$.
Replacing the coordinate function $x_j$ by the operators $x_j\tau_{\mp h\vv_j}$ , we establish
the duality between the finite difference
operators $\partial_{h}^{\pm j}$ and the ``formal'' coordinate
functions $x_j\tau_{\pm h \vv_j}$.

We also note that the operators $x_j\tau_{\pm h\vv_j}$ mutually commute, when acting on functions
\begin{eqnarray}
\label{creation}x_j\tau_{\pm h\vv_j}(x_k\tau_{\pm h\vv_k}f(\underline{x}))=x_k\tau_{\pm h\vv_k}(x_j\tau_{\pm h\vv_j}f(\underline{x})). \end{eqnarray} Furthermore, the commutative
relations (\ref{annihilation}),(\ref{creation}) together with the
duality relation (\ref{duality}) endow an algebraic representation
of the Heisenberg-Weyl (\ref{weylh}) algebra, where the
``formal'' coordinate functions $x_j\tau_{\mp h\vv_j}$ represent
``creation'' operators dual to the ``annihilation'' operators
$\partial_h^{\pm j}$.

Applying the Rodrigues formula (\ref{rodrigues}), we obtain the following multivariate polynomials $$(x)_{\pm}^{(\alpha)}=\prod_{j=1}^n
x_j(x_j\pm h)\ldots(x_j\pm(\alpha_j-1)h).$$

These polynomials are the basic polynomials that satisfy the equations
(\ref{xQ}). It is interesting to see that these multivariate
polynomials coincide with the multi-index factorial powers
$(x)_{\pm}^{(\alpha)}$ introduced by N.~Faustino and U.~K\"ahler in
\cite{FK07}.

Here we would like to emphasize that the operators $x_j\tau_{\pm h\vv_j}$ and $\tau_{\pm h\vv_j}$ does not commute when acting on functions, i.e.
\begin{eqnarray*}
\tau_{\pm h\vv_j}(x_k\tau_{\pm h\vv_k}f(\underline{x}))=\pm h \delta_{jk}f(\underline{x}\pm h(\vv_j+\vv_k))+x_k\tau_{\pm h\vv_k}(\tau_{\pm h\vv_j}f(\underline{x})), \\
\tau_{\mp h\vv_j}(x_k\tau_{\pm h\vv_k}f(\underline{x}))=\mp h \delta_{jk}f(\underline{x}\pm h(\vv_k-\vv_j))+x_k\tau_{\pm h\vv_k}(\tau_{\pm h\vv_j}f(\underline{x}))
\end{eqnarray*}

Another family of polynomials involving the position operator $x_j'$
defined in (\ref{xj'2}) can be considered. In particular, recursive
application of the Rodrigues formula (\ref{rodrigues}) yields the
following family of polynomials.
$$
\prod_{j=1}^n\frac{1}{2^{\alpha_j}}\left(x_j\tau_{\pm h\vv_j}+\tau_{\pm h\vv_j}x_j\right)^{\alpha_j}{\bf 1}
=\prod_{j=1}^n\left(x_j \pm \frac{1}{2}h\right)\left(x_j \pm \frac{3}{2}h\right)\ldots \left(x_j \pm \frac{2\alpha_j-1}{2}h\right).
$$
We will come back to the Heisenberg-Weyl character of the operators
$\partial_{h}^{\pm j}$ and $x_j'$ in Section \ref{umbralBridge}
when establishing the comparison between
 the Umbral version of Almansi-type decomposition and the Fischer Decomposition for Difference Dirac operators obtained in \cite{FK07}. 

\end{example}

\begin{example}[Central difference operator]\label{example2}

In lattice gauge theories it is common to consider the central finite difference operator (i.e. the average between the forward and backward differences)
$$O_{x_j}=\frac{1}{2}(\partial_h^{+j}+\partial_{h}^{-j})=\frac{1}{h}\sinh(h\partial_{x_j})$$
instead of the forward/backward differences $\partial_h^{\pm j}$.

One among many reasons behind this choice is that contrary to $\partial_{h}^{\pm j}$, $O_{x_j}$ are symmetric and self-adjoint on the Hilbert space $\ell_2(h\mathbb{Z}^n)$.

Furthermore the Pincherle derivative for $O_{x_j}$ corresponds to
$$ O_{x_j}'=\frac{\tau_{h \vv_j}+\tau_{-h \vv_j}}{2}=\cosh(h\partial_{x_j}).$$

It can be easily seen that the multivariate operator $O_{\underline{x}}=(O_{x_1},O_{x_2},\ldots,O_{x_n})$ is a delta operator so it admits a basic polynomial sequence satisfying the equations (\ref{xQ}) and that
$O_{x_j}'$ has an inverse since $O_{x_j}'1=1 \neq 0$.

Now it arises the question how to compute $(O_{x_j}')^{-1}$. Since the map $\iota:\BR[\underline{x}] \rightarrow \mathcal{U}$ is an an algebra isomorphism, the problem of finding an inverse for $O_{x_j}'$ it is enough to find an inverse for the series expansion
$$ \cosh(hx_j)=\sum_{k=0}^{\infty}\frac{1}{(2k)!}(hx_j)^{2k}$$

We however note that finding an inverse for $\cosh(hx_j)$ becomes increasingly
cumbersome. Nevertheless, there is a certain systematic in it.

For a sake of simplicity, we impose periodic boundary conditions on our lattice such that $\underline{x}+N\vv_j=\underline{x}$, i.e. $(\tau_{h \vv_j})^N={\bf id}$ for a certain $N \in \BN$.(i.e. a periodic lattice). In particular, when $N/2$ is odd, $(O_{x_j}')^{-1}$ corresponds to c.f. \cite{DHS96}
$$ (O_{x_j}')^{-1}=\sum_{k=0}^{N/2-1} (-1)^k (\tau_{h \vv_j})^k$$
while for $N$ odd one obtain
$$ (O_{x_j}')^{-1}=\sum_{k=0}^{(N-1)/2} (-1)^{k+(N-1)/2} (\tau_{h \vv_j})^{2k}+
\sum_{k=0}^{(N-1)/2-1} (-1)^{k} (\tau_{h \vv_j})^{2k+1}.$$
These inversion formulae will be of special interest on Section \ref{umbralBridge} when we consider the gauged version of the Harmonic oscillator.
\end{example}




\section{Almansi-type theorems in Umbral Clifford analysis}

\subsection{Umbral Clifford Analysis}\label{umbralCliffordAnalysis}

The main objective of this subsection is to introduce a well-adapted
Clifford setting.

In what follows we retain the notation adopted in subsection
\ref{quantumM}.
Additionally we will denote by $\cl_{0,n}$ the algebra of
signature $(0,n)$ whose generators $\e_1,\e_2,\ldots,\e_n$ satisfy
the graded deformed anti-commuting relations
\begin{eqnarray}\label{eq:1.42}
\label{CliffPinch}\{
\e_j,\e_k\}=-2\delta_{jk},
\end{eqnarray}
where $\{{\bf a},{\bf b}\}={\bf a}{\bf b}+{\bf b}{\bf a}$
 denotes the bracket between ${\bf a}$ and ${\bf b}$.

The following algebra is an algebra of radial type, i.e.
\begin{eqnarray}
\label{radial}\left\{~\left[\e_j,\e_k \right],\e_l~\right\}=0, &
\mbox{for all}~j,k,l=1,\ldots,n.
\end{eqnarray}
In the above definition there is indeed no a priori defined a linear
space to which the generators $\e_j$ belong. Nevertheless, by only
using (\ref{radial}) one can already deduce many properties \cite{Sommen97}.

Analogously to \cite{Sommen97}, we define Umbral Clifford analysis
as the study of the operators belonging to the algebra of
differential operators
\begin{eqnarray}\label{eq:1.43}
\mbox{Alg}\left\{
x_j',O_{x_j},\e_j~:~j=1,\ldots,n\right\},
\end{eqnarray}
where the operators $O_{x_j}$ and $x_j'$ are defined accordingly to (\ref{Oxj}) and (\ref{xj'})/(\ref{xj'2}), respectively.

 We define the umbral counterpart for the Dirac operator as
\begin{eqnarray}
\label{dirac}D'=\sum_{j=1}^n \e_j O_{x_j}.
\end{eqnarray}
and by
\begin{eqnarray}
\label{x'}x'=\sum_{j=1}^n \e_j x_j'
\end{eqnarray}
the umbral counterpart of the vector variable.

 The Heisenberg-Weyl character of the operators $x_j'$ and
$O_{x_j}$ together with the commutation relations (\ref{weylh2}), and the anti-commutation rule (\ref{CliffPinch})
implies that
\begin{eqnarray}
\label{xsquare}(x')^{2}=-\sum_{j=1}^n (x_j')^2, \\
\label{Dsquare}(D')^2=-\sum_{j=1}^n O_{x_j}^2.
\end{eqnarray}

The operator $(x')^2$ then corresponds to the generalization of the
norm squared of a vector in the Euclidean space since $|
x'|^2=-(x')^2{\bf 1}$ while the right hand side of (\ref{Dsquare})
corresponds to the umbral counterpart of the Laplacian, i.e.
$\Delta'=-(D')^2$, and this indeed what we expected from a "true"
coordinate variable and Dirac operator.

Bellow there are some examples of this type of operators:
\begin{example}
If $O_{x_j}=\partial_{x_j}$, we recover the standard Clifford Analysis since the umbral operators $D'$ and $x'$
 coincide with the standard Dirac and coordinate variable operators, respectively.
    Furthermore the {\it continuum} Hamiltonian $\frac{1}{2}\left(-\Delta+|\underline{x}|^2 \right)$
     can we rewritten as $$\frac{1}{2}\left(-\Delta+|\underline{x}|^2 \right)=\frac{1}{2}(D^2-x^2),$$
      where the quantity $\frac{1}{2}|\underline{x}|^2=-\frac{1}{2}x^2$ should be interpreted as a spherical symmetric potential operator.
\end{example}

\begin{example}\label{discreteClifford}
If $O_{x_j}=\partial_{h}^{+j}$, where $\partial_h^{+j}$ are the forward differences introduced on Example \ref{example1}, the corresponding operator $D'$ is given by
    \begin{eqnarray}
    \label{DiscreteDirac}D'=\sum_{j=1}^n \e_j \partial_h^{\pm j}.
    \end{eqnarray}
    The square of $(D')^2$ does not split the star laplacian
    \begin{eqnarray}
    \Delta' \neq \sum_{j=1}^n
\partial_{h}^{+j}\partial_{h}^{-j},
 \end{eqnarray}

\end{example}

\begin{example}\label{discreteClifford2}
If $O_{x_j}=\frac{1}{2}\left(\partial_{h}^{+j}+\partial_{h}^{-j}\right)$ are the central differences introduced on Example \ref{example2}, the corresponding operator $D'$ is given by
    \begin{eqnarray}
    \label{DiscreteDirac}D'=\frac{1}{2}\sum_{j=1}^n \e_j \left(\partial_{h}^{+j}+\partial_{h}^{-j}\right).
    \end{eqnarray}
    The square of $(D')^2$ split the star laplacian
    \begin{eqnarray}
    \label{starLapl}\Delta'=\sum_{j=1}^n
\partial_{2h}^{+j}\partial_{2h}^{-j},
 \end{eqnarray}
\end{example}

In Examples \ref{discreteClifford} and \ref{discreteClifford2}, the operator $-\frac{1}{2}(x')^2$ describes a spherical potencial of the lattice while the equation $-(x')^2{\bf 1}=r$ describes a lattice sphere on the $n-$dimensional ambient space.

We would like to point out that under the conditions of (\ref{example2}), the special choice (\ref{xj'2}) leads to spherical potentials on the lattice which possesses the symmetric property like in {\it continuum}. However the same does not holds under the conditions of Example \ref{example1}.

We will come back to this problem in Section \ref{umbralBridge} when we consider the gauged version of the harmonic oscillator.

\subsection{Decomposition of Umbral polymonogenic functions}


\begin{definition}
Let $\Omega$ be a domain in $\BR^n$ and $k\in \BN$.
 A function
$f: \Omega\longrightarrow \cl_{0, n}$ is \textsl{Umbral
~polymonogenic} of degree $k$ if $(D')^kf=0$. If $k=1$, it is called
{\it Umbral monogenic}~function.
\end{definition}

\begin{definition}
A domain $\Omega\subset\mathbf R^n$ is  a \textsl{starlike domain}
with center $0$ if with $x\in \Omega$ also  $tx\in\Omega$ holds
for any $0\le t \le 1$.
\end{definition}

Let $\Omega$ be a starlike domain with center $0$. For any $s>0$, we
define the operator  $I_s: C(\Omega, \cl_{0, n})\longrightarrow
C(\Omega, \cl_{0, n})$ by
\begin{equation}\label{eq1.11}
I_s f(x)=\int_0^1 f(tx) t^{s-1} dt.
\end{equation}

Recalling the definition of $\Psi_{\underline{x}}$ in
(\ref{eq:1.41}), we now introduce an operator
\begin{equation}\label{eq1.12}I_s'=\Psi_{\underline{x}} I_s \Psi_{\underline{x}}^{-1}.\end{equation}

For any $k\in\BN$, denote
 \begin{equation}\label{eq1.21}
Q_k'= \frac{1}{c_k} I_{\frac{n}{2}}' I_{\frac{n}{2}+1}'\cdots
I_{\frac{n}{2}+[\frac{k-1}{2}]}',
 \end{equation}
where $ c_k=(-2)^k [k/2]!.$

\begin{theorem}\label{th:01}
Let $\Omega$ be a starlike domain in $\BR^n$ with center $0$. If $f$
is a Umbral polymonogenic function in $\Omega$ of degree $k$, then
there exist unique functions $f_1, \ldots, f_k$, each Umbral
monogenic in $\Omega$, such that
 \begin{equation}\label{eq1.32}
 f(x)=f_1(x)+ x'f_2(x) + \cdots + (x')^{k-1}f_k(x).
\end{equation}
 Moreover the Umbral monogenic functions
$f_1, \ldots, f_k$ are given by the the following formulas:

\begin{equation}\label{eq1.33}
\begin{array}{rcr}
 f_k &=& Q_{k-1}' (D')^{k-1} f(x)
   \\
f_{k-1}&=& Q_{k-2}' (D')^{k-2}({\bf id}-(x')^{k-1}Q_{k-1}'
(D')^{k-1}) f(x)
    \\
& \vdots
 \\
f_2 &=& Q_1D'({\bf id}-(x')^2Q_2(D')^2)\cdots({\bf
id}-(x')^{k-1}Q_{k-1} (D')^{k-1})f(x)
  \\
f_1 &=& ({\bf id}-x' Q_1' D')({\bf id}-(x')^2Q_2'(D')^2)\cdots({\bf
id}-(x')^{k-1}Q_{k-1} (D')^{k-1}) f(x)\rlap.
\end{array}
\end{equation}

Conversely, the sum in  (\ref{eq1.32}), with $f_1, \ldots, f_k$
Umbral monogenic in $\Omega$, defines a Umbral polymonogenic
function in $\Omega$.
\end{theorem}

\subsection{The proof of the Almansi Theorem}

Before proving Theorem \ref{th:01}, we need some lemmas.

Denote
\begin{equation}\label{eq1.88}
E_{s}=s{\bf id}+E.
\end{equation}
We write $E$ instead of $E_{0}$ when $s=0$.

\begin{lemma}\label{le:4.1}
(\cite{MR02}) Let $x\in \BR^n$ and $\Omega$ be a domain with
$\Omega\supset [0, x]$.  If $s>0$ and $f\in C^1(\Omega,
\cl_{0,n}')$, then
\begin{equation}\label{eq1.880} f(x)=I_s E_{s} f (x)= E_{s} I_s f(x).\end{equation}
\end{lemma}

Note that Lemma \ref{le:4.1}  remains true  in the setting of Umbral
Clifford analysis, since it holds componentwise.

Since $I_s'=\Psi_{\underline{x}} I_s \Psi_{\underline{x}}^{-1}$ and
$E'_s=\Psi_{\underline{x}} E_s \Psi_{\underline{x}}^{-1}$, it
follows from (\ref{eq1.880}) that
$$E'_sI_s'=I_s'E'_s={\bf id}.$$

\begin{lemma}\label{le:51}  Let $\Omega$ be a domain in $\BR^n$.
  If $f\in C^1(\Omega)$, then
\begin{equation}\label{eq3.1}
  D' E'_s f(x)=E'_{s+1} D' f (x), \qquad D' I'_s f(x)=I'_{s+1} D' f
  (x).
 \end{equation}
\end{lemma}

\begin{proof} Recall that $[O_{x_j}, x'_k]=\delta_{jk}{\bf id}$ and $[O_{x_j}, O_{x_k}]=0$, i.e.,
$$O_{x_j} x'_k=\delta_{jk}{\bf id}+ x'_kO_{x_j},  \qquad O_{x_j}O_{x_k}=O_{x_k}O_{x_j}.$$
By definition
$$D'=\sum_{j=1}^n e_j O_{x_j}, \qquad E'=\sum_{k=1}^n x_k'
O_{x_k},$$ so that
$$\begin{array}{rcl}
D'E'&=&\sum_{j, k=1}^n e_j O_{x_j} x'_k O_{x_k}
\\
&=&\sum_{j, k=1}^n e_j (\delta_{jk}{\bf id}+ x'_kO_{x_j}) O_{x_k}
\\
&=&\sum_{j, k=1}^n \delta_{jk}e_j  O_{x_k}+ \sum_{j, k=1}^n e_j
 x'_kO_{x_j} O_{x_k}
\\
&=&\sum_{j}^n e_j  O_{x_j}+ \sum_{j, k=1}^n
 x'_kO_{x_k}  e_j O_{x_j}
\\&=&D'+E'D'
\\&=&E'_1 D'.
\end{array}
$$
With the same approach, we can also show that $x'D'+D'x'=-2E'_{n/2}$
and $E_{s}'x'=x'E_{s+1}'$.

Hence $D'E'_s=D'(s{\bf id}+E')=sD'+E_{1}^{'} D'=E_{s+1} D'$. Since
$E'_s={I'}_s^{-1}$, we  have
$$D' I'_s =I'_{s+1}E_
{s+1}' D' I'_s=I'_{s+1}D'E'_sI'_s= I'_{s+1} D'.$$
\end{proof}

As direct consequence of (\ref{eq3.1}), we find that $I'_s f$ and
$E'_s f$ are monogenic, whenever $f$ is Umbral monogenic. From the
definition of $Q_k'$ in (\ref{eq1.21}), we thus obtain
 \begin{equation}\label{eq3.11}
 Q_k' H_1=H_1.
 \end{equation}

\begin{lemma}
Let $\Omega$ be a starlike domain in $\BR^n$. For any monogenic
function $f$ in $\Omega$,
$$ (D')^k (x')^k Q_k' f(x)= f(x), \quad x\in \Omega.$$
\end{lemma}

\begin{proof}
From (\ref{eq1.21}) and Lemma \ref{le:51}, we know that
\begin{equation}\label{eq3.111}
(Q_k')^{-1}=c_k E'_{\frac{n}{2}+[\frac{k-1}{2}]}\cdots
E'_{\frac{n}{2}+1} E'_{\frac{n}{2}}. \end{equation}

Denote $g= Q_k' f$. From  (\ref{eq3.11}) and  (\ref{eq3.111}),
 $g$ is Umbral monogenic in $\Omega$ and
$$f(x)=c_k E'_{\frac{n}{2}+[\frac{k-1}{2}]}\cdots
E'_{\frac{n}{2}+1} E'_{\frac{n}{2}} g(x).$$

We need to show
\begin{equation}\label{eq3.12}
(D')^k (x')^k g(x)=c_k E'_{\frac{n}{2}+[\frac{k-1}{2}]}\cdots
E'_{\frac{n}{2}+1} E'_{\frac{n}{2}} g(x),
\end{equation}
 or rather
\begin{equation}\label{eq3.19}
\begin{array}{rcl}
(D')^{2m-1}(x'^{2m-1}) g(x)&=& -(m-1)! 2^{2m-1}
E'_{\frac{n}{2}+m-1}\cdots E'_{\frac{n}{2}+1} E'_{\frac{n}{2}} g(x),
   \\
(D')^{2m}(x'^{2m}) g(x) &=& m! 2^{2m}  E'_{\frac{n}{2}+m-1}\cdots
E'_{\frac{n}{2}+1} E'_{\frac{n}{2}} g(x)
\end{array}
\end{equation}
for any $g$ Umbral monogenic in $\Omega$ and $m\in\BN$.

We use induction to prove (\ref{eq3.19}). Since
$x'D'+D'x'=-2E'_{n/2}$ and $Dg(x)=0$, we have
\begin{equation}\label{eq3.29}
D'(x' g(x))=-2 E'_{\frac{n}{2}} g(x).
\end{equation}

Next we show  that, for any $x\in\Omega$ and $m\in\BN$,

\begin{equation}\label{eq3.39}
\begin{array}{rcl} D' ((x')^{2m} g(x)) &=& -2m (x')^{2m-1} g(x);
\\
D' ((x')^{2m-1} g(x)) &=& -2 (x')^{2(m-1)}  E_{\frac{n}{2}+m-1}' g(x).
\end{array}
\end{equation}
 This can be checked by induction. Assuming that (\ref{eq3.39})
 holds for $m$. we shall now prove it also holds for $m+1$.
We now apply the operator  $x'D'+D'x'=-2E'_{n/2}$ to the function
$(x')^{2m}g(x)$:
\begin{equation}\label{eq3.191}
x'D'((x')^{2m}g(x))+D'x'((x')^{2m}g(x))=-2E'_{n/2}((x')^{2m}g(x)).
\end{equation}
By the hypothesis of induction, the first term in the left is equal
to $-2m(x')^{2m}g(x)$, while the left side equals
$-2x'^{2m}E'_{n/2+2m} g(x)$ due to the fact that
$$E'_{s}x'=x'E'_{s+1}.$$
As a result,
\begin{equation}\label{eq3.392}
\begin{array}{rcl} D'(x'^{2m+1}g(x))&=&-x'D'(x'^{2m}g(x))-2E'_{n/2}(x'^{2m}g(x))
\\&=&2mx'^{2m}g(x)-2 x'^{2m} E'_{n/2+2m} g(x)
\\&=&-2 x'^{2m} E'_{n/2+m} g(x).
\end{array}
\end{equation}
This proves the second equality of (\ref{eq3.39}). The first
equality of (\ref{eq3.39}) can be proved similarly. This proves the
identities (\ref{eq3.39}).

Finally we apply (\ref{eq3.39}) to prove (\ref{eq3.19}). It is easy
to prove when $m=1$. For the general case, from (\ref{eq3.39}) we
have
$$
\begin{array}{rcl}
(D')^{2m} ( (x')^{2m} g(x)) &=& (D')^{2m-1} D'((x')^{2m} g(x))
\\
&=&  -2m (D')^{2(m-1)} ( (x')^{2m-1} g(x))
\end{array}
$$
and
$$
\begin{array}{rcl}
(D')^{2m-1} ( (x')^{2m-1} g(x)) &=& (D')^{2(m-1)-1} D'(D'((x')^{2m-1} g(x))
)
\\
&=&  (D')^{2(m-1)-1} D'(-2 (x')^{2(m-1)} E'_{\frac{n}{2}+m-1} g(x))
\\
&=&  4(m-1) (D')^{2(m-1)-1}  ( (x')^{2(m-1)-1} E'_{\frac{n}{2}+m-1} g(x)
).
\end{array}
$$
(\ref{eq3.19}) follows directly from the above induced formulas
and the assumption of induction.
\end{proof}

Now we come to the proof of our main theorem.

\begin{proof}[of Theorem \ref{th:01}]
Recall $H_k=\{f\in C^\infty(\Omega, \cl_{0, n}): (D')^k f=0 \}$. It is
sufficient to show that
$$
H_k=H_{k-1}+T_{k-1} H_1, \quad k\in\BN,
$$
where $T_k=(x')^k$. Notice that Lemma 3.3 states that
\begin{equation}\label{eq3.3}
 (D')^k T_k Q_k'={\bf id}.
\end{equation}

We divide the proof into two parts.

\begin{enumerate}

\item[(i)] \ $H_k\supset H_{k-1}+T_{k-1} H_1$.

Since $H_{k-1}\subset H_k$, we need only to show $T'_{k-1}
H_1\subset H_k$. For any $g\in H_1$, by (\ref{eq3.3}) and
(\ref{eq3.11})  we have
\begin{eqnarray*}
 (D')^k (T_{k-1}g) = D' ((D')^{k-1} T_{k-1} Q_{k-1} ) (Q_{k-1}')^{-1} g
= D' (Q_{k-1}')^{-1} g=0.
 \end{eqnarray*}

\item[(ii)]
\ $H_k\subset H_{k-1}+T_{k-1} H_1$.

For any $f\in H_k$, we have the decomposition
$$f=({\bf id}-T'_{k-1} Q_{k-1}' D'^{k-1} )f + T_{k-1} (Q_{k-1}' (D')^{k-1} f).$$
We will show that the first summand above is in $H_{k-1}$ and the item in
the braces of the second summand is  in $H_1$. This can be verified
directly. First,
\begin{eqnarray*}
 (D')^{k-1}({\bf id}-T_{k-1} Q_{k-1} (D')^{k-1} )f&=&
((D')^{k-1}- ((D')^{k-1} T_{k-1} Q_{k-1}') (D')^{k-1} )f \\
&=& ((D')^{k-1}-  (D')^{k-1} )f=0.
 \end{eqnarray*}
Next, since $(D')^{k-1}f\in H_1$ and $Q_{k-1}' H_1\subset H_1$, we
have $ Q_{k-1}' (D')^{k-1} f\in H_1$, as desired.

\end{enumerate}

\noindent This proves that $H_k=H_{k-1}+T_{k-1} H_1$. By induction,
we can easily deduce that $H_k=H_1+ T_1H_1+\cdots+ T_{k-1} H_1$.

\medskip
Next we prove that for any $f\in H_k$ the decomposition
$$ f= g+T_{k-1}f_k, \quad g\in H_{k-1}, f_k\in H_1$$
is unique. In fact, for such a decomposition, applying $(D')^{k-1}$
on both sides we obtain
\begin{eqnarray*}
(D')^{k-1} f &=& (D')^{k-1} g+(D')^{k-1} T_{k-1} f_k
\\
&=& (D')^{k-1} T_{k-1} Q_{k-1}' (Q_{k-1}')^{-1}  f_1
\\
&=& (Q_{k-1}')^{-1} f_k.
\end{eqnarray*}
Therefore
$$ f_k=Q_{k-1}'(D')^{k-1} f,$$
so that
$$g=f-T_{k-1}f_k=({\bf id}-T_{k-1}Q_{k-1}'(D')^{k-1}) f.$$
Thus (\ref{eq1.31}) follows by induction.

To prove the converse, we see from (\ref{eq3.12}) that, for any
$k\in \BN$,
 $(D')^{k+1} (x')^k H_1=0.$
Replacing $k$ by $j$,  we have
$$
(D')^{k} (x')^j H_1=0$$ for any $k>j$.
 \end{proof}

\section{The Quantum Harmonic Oscillator}

In this section, for a set of creation and annihilation operators $\a_j^+$ and $\a_j^-$, respectively, we shall decompose the Hamiltonian of the classical harmonic oscillator $$\mathcal{H}=\frac{1}{2}\sum_{j=1}^n -(\a_j^-)^2+(\a_j^+)^2$$
into two first order operators $D_+$ and $D_-$ .
 The operators $D_+$ and $D_-$ constitute a
 pair of the same roles as a vector operator involving position and a momentum operators, respectively. We thus can
establish the Almansi decomposition with respect to this
pair similar to our main result in the Almansi decomposition for the
ordinary pair.

\subsection{Classical Hamiltonian splitting and Umbral Calculus}\label{ClassicalHamiltonian}

We first consider the classical case.  We set the Hamiltonian
operator $\mathcal{H}=\frac{1}{2}(-\Delta+|\underline{x}|^2)$ and put
$\a_j^{\pm}=\frac{1}{\sqrt{2}}(x_j \mp \partial_{x_j})$.

As the operators $\a_j^\pm$ satisfy the Heisenberg-Weyl commutation
relations (\ref{weylh}), we can split $\mathcal{H}$ as
$$ \mathcal{H}=E^{+-}+\frac{n}{2}{\bf id},$$

where $E^{+-}:=\sum_{j=1}^n \a_j^+ \a_j^-$ is Hamiltonian of a field
of free (non-interacting) bosons. This operator plays the same role
of the Umbral Euler operator (\ref{umbralE}).

 For $D=\sum_{j=1}^n \e_j \partial_{x_j}$ and
$x=\sum_{j=1}^n \e_j x_j$, where $\e_1,\e_2,\ldots,\e_n$ are the
standard Clifford generators satisfying the anti-commuting relations
$$ \{\e_j,\e_k\}=-2\delta_{jk}{\bf id},$$
we set $D_{\pm}=\frac{1}{\sqrt{2}}(x \mp D)=\sum_{j=1}^n \e_j
\a_j^\mp$. Hence the operators $D_\pm$ and $\mathcal{H}$ satisfy the following relations
\begin{eqnarray}
\label{DH1}\{ D_+, D_-\}=-2\mathcal{H}, &
[H,D_-]=-D_-, &
[H,D_+]=D_+.
\end{eqnarray}
Using the above relations one may easily calculate the following commutators
\begin{eqnarray}
\label{DH2}\left[ D_+^2, D_-^2\right]=4\mathcal{H}, &
\left[\mathcal{H},D_-^2\right]=-2D_-^2, &
\left[\mathcal{H},D_+^2\right]=2D_+^2, \\
\label{DH3}\left[ D_+^2, D_+\right]=0, &
\left[D_-^2,D_+\right]=2D_-, &
\left[\mathcal{H},D_+^2\right]=2D_+, \\
\label{DH4} &\left[D_-,D_+^2\right]=-2D_+. &
\end{eqnarray}
From the relations (\ref{DH2}), we conclude that $\frac{1}{2}D_+^2$, $\frac{1}{2}D_-^2$ and $\mathcal{H}$ are the
canonical generators of the Lie algebra $sl_2(\BR)$ and hence we have a
representation of harmonic analysis for the classical Harmonic
oscillator (see e.g. \cite{HT92}). On the other hand, there is such a kind of
Clifford Analysis framework involving the operators $D_+$ and $D_-$
since
$$\mbox{span}\left\{\frac{1}{2}D_-^2,\frac{1}{2}D_+^2,\mathcal{H}\right\}\oplus
\mbox{span}\left\{D_-,D_+\right\}$$ equipped with the standard
graded commutator $\left[ \cdot, \cdot\right]$ is a isomorphic to a Lie superalgebra
of type $osp(1|2)$ (see e.g. \cite{FSS00}).

\section{Almansi-type decomposition for the Umbral Harmonic Oscillator}

Let us consider the umbral counterpart of the Hamiltonian written in the language of Umbral Clifford Analysis corresponds to
$$\mathcal{H}'=\frac{1}{2}\left(~(D')^2-(x')^2~\right).$$

where $D'$ and $x'$ are Umbral Dirac and the Umbral coordinate variable defined in equations (\ref{dirac}) and (\ref{x'}), respectively. Next we put
\begin{eqnarray}
\label{Dpm}D_\pm':= \frac{1}{\sqrt{2}}(x' \mp D')
\end{eqnarray}

The representations of the Lie algebra $sl_2(\BR)$ and the Lie
superalgebra $osp(1|2)$ also fulfil in the umbral setting. Indeed,
if we consider the Sheffer operator $\Psi_{\underline{x}}$ defined
in (\ref{eq:1.41}), the umbral operators
$$\frac{1}{\sqrt{2}}(x_j'
\mp O_{x_j})=\Psi_{\underline{x}}\a_j^{\pm}\Psi_{\underline{x}}^{-1}$$ satisfy the Heisenberg-Weyl commutation relations
(\ref{weylh}).

Moreover since $D_{\pm}'=\Psi_{\underline{x}}D_\pm\Psi_{\underline{x}}^{-1}$ and $\mathcal{H}'=\Psi_{\underline{x}}\mathcal{H}\Psi_{\underline{x}}^{-1}$, the
relations (\ref{DH1})-(\ref{DH4}) also fulfils for the operators $D_{\pm}'$, $(D_\pm')^2$ and $\mathcal{H}'$. Now we are in conditions to formulate the Almansi-type decomposition for the Umbral Harmonic Oscillator.

Here the main step consists in replace the Umbral Dirac operator
 $D'$ by $D_-'$ and the umbral coordinate vector variable $x'$ by $D_+'$.

 We take the definition of $\Psi_x$,
$I_s'$, $Q_k'$ as in (\ref{eq:1.41}), (\ref{eq1.12}) and
(\ref{eq1.21}) by replacing the operators $x_j'$, $E_{s}'$ and $I_{s}'$ by $\a_j^+$, $(s-\frac{n}{2}){\bf id}+\mathcal{H}'$ and $((s-\frac{n}{2}){\bf id}+\mathcal{H}')^{-1}$, respectively.

Hence we can deduce that if $D_-'g(x)=0$
then
$$D_-'( (D_+')^{2m}g(x))=-2m( D_+')^{2m-1}g(x),\qquad D_-'( (D_+')^{2m-1}g(x))=-2(
D_+')^{2(m-1)}g(x);$$ Moreover
$$D_-'^{k}( (D_+')^{k}g(x))=g(x).$$

Let $\Omega$ be a domain in $\BR^n$ and $k\in \BN$ and let
$H_k^-(\Omega)$ denote the set of all  functions $f:
\Omega\longrightarrow \cl_{0, n}$ such that $(D_-')^kf=0$. With the
same proof as in Theorem \ref{th:01}, we have the following Almansi
decomposition in terms the eigenvectors of the Umbral Hamiltonian $\mathcal{H}'$.

\begin{theorem} \label{harAlmansi}
Let $\Omega$ be a starlike domain in $\BR^n$ with center $0$. If
$f\in H^-(\Omega)$, then there exist unique functions $f_1, \ldots,
f_k$, each in $H_1^-(\Omega)$, such that
 \begin{equation}\label{eq1.309}
 f(x)=f_1(x)+ D_+'f_2(x) + \cdots + (D_+')^{k-1}f_k(x).
\end{equation}
 Moreover
$f_1, \ldots, f_k$ are given by the the following formulas:

\begin{equation}\label{eq1.319}
\begin{array}{rcr}
 f_k &=& Q_{k-1}' (D'_-)^{k-1} f(x)
   \\
f_{k-1}&=& Q_{k-2}' (D'_-)^{k-2}({\bf id}-(D_+')^{k-1}Q_{k-1}'
(D'_-)^{k-1}) f(x)
    \\
& \vdots
 \\
f_2 &=& Q_1'D'_-({\bf id}-(D_+')^2Q_2'(D'_-)^2)\cdots({\bf
id}-(D_+')^{k-1}Q_{k-1}' (D'_-)^{k-1})f(x)
  \\
f_1 &=& ({\bf id}-D_+' Q_1' D'_-)({\bf
id}-(D_+')^2Q_2'(D'_-)^2)\cdots({\bf id}-(D_+')^{k-1}Q_{k-1}' (D'_-)^{k-1})
f(x)\rlap.
\end{array}
\end{equation}

Conversely, the sum in  (\ref{eq1.309}), with $f_1, \ldots, f_k\in
H_1^-(\Omega)$, defines a function in $H_k^-(\Omega)$.
\end{theorem}

\section{Umbral Bridge} \label{umbralBridge}

We can now show that the powerful "Umbral bridge" connects
continuity and discreteness.

First we consider the continuous case. Namely,  we take $O_{x_j}$ to
be the usual continuous derivative $\partial_{x_j}$. Then the
Pincherle derivative $(\partial_{x_j})'=[\partial_{x_j}, x_j]={\bf id}$, $D'$ is exactly
the classical Dirac operator, so that the resulting Umbral Clifford
algebra turns out to be the usual Clifford algebra $\cl_{0,n}$ and hence our
main theorem then recoveries the Almansi decomposition as in
\cite{MR02}.

It is now natural to take a look for applications of the umbral calculus formalism in problems related to lattices. Our main objective in this section is therefore to find a Almansi-type decomposition for the gauged version of the Dirac Operator and for the quantum harmonic oscillator.

First we will explain how to recover the Almansi-type decomposition in the discrete case.
\begin{example}[Discrete counterpart of Almansi Decomposition]\label{discreteAlmansi}

Our starting point is to consider the difference Dirac operators $D_{h}^{\pm}$ introduced in \cite{FK07}.
In the terminology of that paper, $$D_h^{\pm}=\sum_{j=1}^n \e_j\partial_h^{\pm j}$$ are the forward/backward versions of the Dirac operator, where $\partial_h^{\pm j}$ introduced in Example \ref{example1} and $\e_1,\e_2,\ldots,\e_n$ being the standard basis for the Clifford Algebra $\cl_{0,n}$.

By introducing the Sheffer maps $$\Psi_{\underline{x}}^{\pm}: \underline{x}^\alpha \mapsto (\underline{x})^{(\alpha)}_{\mp},$$ the Almansi-type decomposition for the operators can be obtained by replacing the standard Dirac $D=\sum_{j=1}^n \e_j \partial_{x_j}$ and coordinate variable operators $x=\sum_{j=1}^n\e_j x_j$ by its discrete counterparts
\begin{eqnarray}
D_h^{\pm}=\Psi_{\underline{x}}^{\pm}D (\Psi_{\underline{x}}^{\pm})^{-1}, & \mbox{and}~~
x^{\pm}=\Psi_{\underline{x}}^{\pm}x(\Psi_{\underline{x}}^{\pm})^{-1}=\sum_{j=1}^n\e_j ~x_j\tau_{\mp h \vv_j},
\end{eqnarray}
respectively (see \cite{MR02}).

\end{example}

A further consequence, we recover the Fischer Decomposition obtained in \cite{FK07}. If
 $\Pi_k$ denote the space of all Clifford-Valued polynomials homogeneous of degree $k$, and $\Pi$ corresponds to the graded algebra $$\Pi=\bigoplus_{k=0}^\infty \Pi_k,$$ the set
 $$ \Pi^{\mp}=\Psi_{\underline{x}}^{\pm}\Pi_k=\left\{ f \in \Pi : E_h^{\mp}f=kf\right\}.$$
corresponds to the spaces of Clifford-valued polynomials generated by the polynomials $(\underline{x})_{\mp}^{(\alpha)}$ of degree $|\alpha|=k$, where $E^{\mp}_h=\sum_{j=1}^n x_j \partial^{\mp j}_h$ stands the diference Euler operators introduced in \cite{FK07}.

Then the Fischer Decomposition for the Difference Dirac Operators $D_h^\pm$ corresponds to
\begin{eqnarray}
\label{FischerDecomp}\Pi^{\mp}_k=\bigoplus_{j=0}^k (x^{\pm})^j\mathcal{M}^{\mp}_{k-j}, & \mbox{where}~\mathcal{M}^{\mp}_{k-j}:=\Pi^{\mp}_k \cap \ker D_h^{\pm}.
\end{eqnarray}

We have to stress that the Fischer Decomposition obtained above does not allow to obtain a refinement of the discrete spherical harmonics (i.e., the null solutions of the star Laplacian/D'Alembertian (\ref{starLapl}) ) in terms of null solutions of $D_h^\pm$. Now we will turn our attention for the discrete case described in Example \ref{discreteClifford}.

\begin{example}[Refinement of discrete spherical harmonics]
We shall now consider the central difference Dirac operator acting in lattices with mesh-size $h/2$, i.e.
$$D_{h/2}:=\frac{1}{2}\left(D_{h/2}^+ + D_{h/2}^-\right).$$

In this way, taking the central finite differences and the resulting inverse for the Pincherle derivative described in Example \ref{example2}, we build the Sheffer map
\begin{eqnarray}
\label{ShefferMap}\Psi_{\underline{x}}: \underline{x}^\alpha \mapsto \prod_{j=1}^n (x_j')^{\alpha_j}{\bf 1},
 \end{eqnarray}
 {\it viz} the symmetrized formula (\ref{xj'2}) for the operators $x_j'$ and hence, we obtain like in Example \ref{discreteAlmansi} the Almansi-type decomposition as well as the Fischer decomposition for the central difference Dirac operator $D_{h/2}$.
The main advantage of the described approach we split the star Laplacian operator as the square of $D_{h/2}$.

Now we can construct the discrete counterpart of the spherical harmonics into the null solutions of $D_{h/2}$. This is essentially obtained in the following way.

Namely, if $f$ is harmonic on the unit ball $\mathbb{B}:=\{x ~:~ -x^2 \leq 1\},$ there exists a unique decomposition (see \cite{MR02}, page 1545)
\begin{eqnarray}
\label{sphericalH} f(x)=f_1(x)+xf_0(x)
\end{eqnarray}
where $f_1$ and $f_0$ are monogenic in $\mathbb{B}$. Moreover, $f_1$ and $f_0$ are explicitly given by
\begin{eqnarray*}
f_1(x)=({\bf id}+x I_{n/2}D)f(x), & f_0(x)=-I_{n/2}D f(x)
\end{eqnarray*}
Since the Sheffer map is an algebra isomorphism, we obtain the decomposition of discrete harmonic functions on the discrete ball $$\Psi_{\underline{x}}\mathbb{B}:=\left\{x_{h/2}=\Psi_{\underline{x}}x ~:~ -\left(x_{h/2}\right)^2{\bf 1} \leq 1\right\},$$ by applying the Sheffer map $\Psi_{\underline{x}}$ on both sides of (\ref{sphericalH}).
Indeed we have
\begin{eqnarray}
\Psi_{\underline{x}}f(x)=\Psi_{\underline{x}}f_1(x)+x_{h/2}\Psi_{\underline{x}}f_0(x).
\end{eqnarray}
Moreover if we restrict the above decomposition to the discrete sphere
$\partial (\Psi_{\underline{x}}\mathbb{B})$, $\Psi_{\underline{x}}f(x) \in \Psi_{\underline{x}}\Pi_k$ is the discrete counterpart of the $k-$spherical harmonic polynomial $f(x)$ while $\Psi_{\underline{x}}f_1(x) \in \Psi_{\underline{x}}\Pi_k$ is the discrete counterpart of the $k-$spherical monogenic polynomial $f_1(\underline{x})$.
\end{example}

From the definition, we can also see that contrary to the operators $D_{h}^{\pm}$, the central difference Dirac operator $D_{h/2}$ is not a `local operator' since the shift operators $\tau_{\pm \frac{h}{2}\vv_j}$ when acting on some lattice functions, they not only concern with the nearest neighbor points but also all the points contained in the direction $h \vv_j$. So we can call the combination $D_{h/2}$ a `quasi-local' operator. The quasi-local is the price for getting the above mentioned nice description of discrete spherical monogenics.

The next important example concerns the Gauged version of Harmonic Oscillator written in terms of the operators described in the above example.

\begin{example}[The gauged version of the Quantum Harmonic Oscillator]
Now we consider the Gauged version of the Quantum Harmonic Oscillator.

$$\mathcal{H}_h=\Psi_{\underline{x}}\mathcal{H}\Psi_{\underline{x}}^{-1}=\frac{1}{2}\left(D_{h/2}\right)^2-\frac{1}{2}\left(x_{h/2}\right)^2.$$

This local operator is defined on the uniform lattice with mesh-width $h$ and corresponds to a splitting in terms of `quasi-local' operators.

In order to describe the eigenfunctions of the gauged Hamiltonian, we turn now our attention to the continuous harmonic oscillator described in subsection \ref{ClassicalHamiltonian}.

In the continuum, a straightforward computation shows that $$V_\alpha(\underline{x})=\exp(-|\underline{x}|^2)H_\alpha(\underline{x}),$$ are eigenfunctions of $\mathcal{H}$ with corresponding eigenvalue $\lambda=|\alpha|+\frac{n}{2}$, where $H_\alpha(\underline{x})$ stands the classical multivariate Hermite polynomials.

For these eigenfunctions, we also obtain following raising and lowering properties for the set of operators $\a_j^\pm:$
\begin{eqnarray}
\a_j^+ V_\alpha(\underline{x})= V_{\alpha+\vv_j}(\underline{x}), & \a_j^- V_\alpha(\underline{x})= \alpha_jV_{\alpha-\vv_j}(\underline{x}).
\end{eqnarray}
Moreover, the exponential generating function for $V_\alpha(\underline{x})$ is given by
$$
V(\underline{x},\underline{t})=\sum_{\alpha}\frac{V_\alpha(\underline{x})}{\alpha!}\underline{t}^\alpha
=\exp\left(-\frac{1}{2}\left(|\underline{x}|^2+|\underline{t}|^2\right)+\sqrt{2}\underline{x}\cdot \underline{t}\right).
$$
The corresponding eigenfunctions and exponential generating function for the gauged Hamiltonian $\mathcal{H}_h$ are then obtained viz the Sheffer map (\ref{ShefferMap}). These are closely related with the Kravchuk polynomials of discrete variable introduced in \cite{ML01}.

Furthermore for the spaces of Clifford valued polynomials generated by $\Psi_{\underline{x}}V_\alpha(\underline{x})$ we obtain the following version of the Fischer Decomposition
\begin{eqnarray}
\label{FischerDecompGauge}\bigoplus_{j=0}^k (x_{h/2})^{j}\Psi_{\underline{x}}{H}_{k-j}^+(\Omega),
\end{eqnarray}
where $\Psi_{\underline{x}}H_{k-j}^+(\Omega):=\left\{ f \in \Pi: \mathcal{H}_h f=\left(k-j+\frac{n}{2}\right)f \right\}\cap \ker ~(D_{h/2})_-$.
\end{example}

In the last example, the Fischer decomposition results as a direct consequence of the Howe dual pair technique (see \cite{Howe85}).

Moreover, the following theorem appears as a generalization of (\ref{FischerDecompGauge}), allow us to split the eigenvectors of the Gauge Hamiltonian $\mathcal{H}_{h}$ in terms of null solutions for the operator $(D_{h/2})_-$.

\begin{theorem}[Almansi-type decomposition for the Gauge Harmonic Oscillator] \label{discreteAlmansi}
Let $\Omega$ be a starlike domain in $\BR^n$ with center $0$. If $f$ satisfies
\begin{eqnarray}
\label{polym}(D_{h/2})_{-}f(x)=0, &\mbox{and}~~\mathcal{H}_{h}f=(k+\frac{n}{2})f(x)
\end{eqnarray}
in $\Omega_{h/2}:=\Omega \cap (h/2) \mathbb{Z}^n$, then
there exist unique functions $f_1, \ldots, f_k \in \ker (D_{h/2})_-$, such that
 \begin{equation}\label{eq1.301}
 f(x)=f_1(x)+ (D_{h/2})_+f_2(x) + \cdots + (D_{h/2})_+^{k-1}f_k(x).
\end{equation}
 Moreover the discrete  monogenic functions
$f_1, \ldots, f_k$ are given by the the following formulas:

\begin{equation}\label{eq1.31}
\begin{array}{rcr}
 f_k &=& Q_{k-1}' (D_{h/2})_-^{k-1} f(x)
   \\
f_{k-1}&=& Q_{k-2}' (D_{h/2})_-^{k-2}({\bf id}-(D_{h/2})_+^{k-1}Q_{k-1}'
(D_{h/2})_-^{k-1}) f(x)
    \\
& \vdots
 \\
f_2 &=& Q_1'(D_{h/2})_-({\bf id}-(D_{h/2})_+^2Q_2(D_{h/2})_-^2)\cdots({\bf
id}-\left(D_{h/2}\right)_+^{k-1}Q_{k-1} (D_{h/2})_-^{k-1})f(x)
  \\
f_1 &=& ({\bf id}-(D_{h/2})_+ Q_1' (D_{h/2})_-)({\bf id}-(D_{h/2})_+^2Q_2'(D_{h/2})_-^2)\cdots({\bf
id}-(D_{h/2})_+^{k-1}Q_{k-1}' D_{h/2})_-^{k-1} f(x)\rlap.
\end{array}
\end{equation}

Conversely, the sum in (\ref{eq1.301}), the functions $f_1, \ldots, f_k$
satisfy the equations (\ref{polym}) in $\Omega_{h/2}$.
\end{theorem}

With the above results show the powerful Umbral bridge unifies continuous and discrete.

\section{Concluding Remarks:}

Umbral Calculus is an analysis of correspondence between continuity and discreteness in the commutative field related to polynomials. In the article,  we bring Umbral Calculus into the non-commutative field through the Umbral Clifford analysis and we further consider functions far beyond polynomials. We study the discrete version in the Umbral Clifford analysis by keeping the symmetry of  phenomena with the price of non-commutative. The Almansi decomposition is established in Umbral Clifford analysis, which concerns the theory of the powers of Laplacian. In particular, it gives the Fisher decomposition of polynomials, which is a start pointing for the spherical harmonics and special function theory as a whole.  One can further study the theory of Fourier transformations and wavelets.  Recently the theory of power operators of Laplacian has found their applications  in the conformal Geometry in pure mathematics (see \cite{HS01}).

As a  framework of the unification of continuity and discreteness, the strong power of this  framework lies at its constructive isomorphism. This makes it possible to construct the associated  equations and solutions in the discrete setting starting from the equations and their solutions in the continuous setting in Physics. One may  further bring this technique to the field of Geometry by studying the fruitful equations in the continuous setting through the associated discrete version.

\sloppy


\begin{thebibliography}{}

\bibitem{BLR98}{ Di Bucchianico A., Loeb, D.E. and Rota, G.C.
{\it Umbral calculus in Hilbert space} In: B. Sagan and R.P. Stanley
(eds.), Mathematical Essays in Honor of Gian-Carlo Rota , pp.
213-238, Birkhäuser, Boston, 1998.}

\bibitem{roman84}{ Roman - {\it The Umbral Calculus}, Academic Press, San Diego (1984)}

\bibitem{RR}
Roman, S. and Rota, G. C. `The umbral calculus', Adv.  Math.
27(1978), 95-188.

\bibitem{BDS}
Brackx, F., Delanghe and Sommen F., `Clifford Analysis', Pitman
Publishers, Boston-London Melbourne, 1982.


\bibitem{DSS}
Delanghe, R., Sommen, F. and Soucek, V., Clifford algebra and
spinor-valued functions. Amsterdam: Kluwer Acad. Publ., 1992.


\bibitem{GM}
Gilbert J.  and  Murray M. `Clifford Algebra and Dirac Operators in
Harmonic Analysis,' Cambridge University Press, Cambridge, 1991.


\bibitem{GS}
G\"urlebeck,K., Spr\"ossig, W., Quaternionic analysis and elliptic
boundary value problems, Berlin: Akademie - Verlag, 1989.

\bibitem{HT92}
Howe,R., Tan, E., Nonabelian harmonic analysis, Universitext. Springer-Verlag, New York, 1992. Applications of $SL(2,\BR)$.s


\bibitem{LV} Lurie, S. A., Vasiliev, V. V.,  The Biharmonic
Problem in the Theory of Elasticity, Amsterdam: Gordon and Breach
Pub.
 1995, 1--265.


\bibitem{AEG} Andersson, L. E., Elfving, T., Golub G. H., Solution
of biharmonic equations with application to radar imaging, J.
Comput. Appl. Math. 1998, 94(2): 153--180.



\bibitem{AGH} Armitage, D., Gardiner, St., Haussmann, W., Characterization of best harmonic and subharmonic
$L_1$-approximation, J. Reine Angew. Math. 1996, 478: 1-15.


\bibitem{K} Kounchev, O. I.,  Multivariate Polysplines: Applications to
Numerical and Wavelet Analysis,  San Diego: Academic Press 2001,
1--498.


\bibitem{ALM}
Almansi, E., `Sulle integrazione dell` equazione differenziale
$\Delta ^{2m}u=0$, Ann. Mat. Pura Appl. Suppl. 3, \textbf{2.},
(1898).


\bibitem{ACL}
Aronszajn, N., Creese,T. M. and Lipkin, L. J., Polyharmonic
functions,  Oxford Mathematical Monographs, Oxford Science
Publications. The Clarendon Press, Oxford University Press, New
York, 1983.


\bibitem{CCGS}
Cohen J. M., Colonna F., Gowrisankaran  K., and  Singman D.,
Polyharmonic functions on trees, Amer. J. Math 124(2002), 999-1043.


\bibitem{R_AD}
Ren, G, `Almansi decomposition for Dunkl operators', Science in
China 48A Supp.(2005), 333-342.

\bibitem{R_U}
Ren, G and K\"ahler U., `Almansi decomposition for polyhamonic,
polyheat, and polywave functions', Studia Math. 172(2006), 91-100.


\bibitem{MAGU}
Ren, G. and  Malonek, H. `Decomposing kernels of iterated
Operators-a unified approach' , Math. Meth. Appl. Sci. 30(2007),
1037-1047.


\bibitem{MR02}
 Malonek, H. and Ren, G,  `Almansi type theorems in Clifford analysis',  Math. Meth. Appl. Sci. 25(2002),
1541-1552.


\bibitem{rota79}{ Joni S.  and Rota G-C - {\it Coalgebras and algebras in combinatorics}, Stud. Appl. Math., 61, (1979), pp. 93-139  }



\bibitem{BL96}{ Di Bucchianico A. and Loeb, D.E.
{\it Operator expansion in the derivative and multiplication by x}, Integral Transf. Spec. Func. 4 (1996), 49-68.}

\bibitem{rota73}{ Rota G-C, Kahane D and Odlyzko A - {\it Finite operator calculus},   J.
Math. Anal. Appl 42 (1973) }


\bibitem{Sommen97}{Sommen F.  {\it An Algebra of Abstract vector variables}, Portugaliae Math. 54, 3 (1997), 287-310}



\bibitem{Vaz97}{
Vaz  J. - {\it  Clifford-like Calculus over Lattice},  Adv. Appl.
Clifford Alg. 7, No.~1 (1997) pp. 37-70 }


\bibitem{FK07}
{N. Faustino, U.~K\"ahler, \textit{Fischer Decomposition for
Difference Dirac Operators}, Adv. Appl. Cliff. Alg., \textbf{17}, no.~1  (2007), 37-58.}


\bibitem{nfaust08}
{ Faustino N., \textit{Rediscovering Clifford Analysis: On the
Interplay between Umbral Calculus and Quantum Mechanics}, in
preparation}

\bibitem{FSS00}{  Frappat L.,  Sciarrino  A.,  Sorba  P.- {\it Dictionary of Lie algebras and super algebras}, Academic Press, New York, 2000 }


\bibitem{Howe85}{ Howe  R.- {\it Dual pairs in physics: harmonic oscillators, photons, electrons, and singletons}, in M.~Flato et al. (Eds), Applications of groups theory in physics and mathematical physics, AMS (1985) pp.179-208}

\bibitem{DHS96}{
 Dimakis A.,  Mueller-Hoissen F.,  Striker T.- {\it  Umbral calculus,
discretization, and quantum mechanics on a lattice},  J. Phys. A 29
(1996) pp. 6861-6876 }



\bibitem{DB}
Delanghe, R., Brackx, F., `Hypercomplex function theory and Hilbert
modules with reproducing kernel',  \emph{Proc. London Math. Soc. (3}
\textbf{37}, No. 3, 545--576,
 (1978).

\bibitem{ML01}
{M. Lorente, \textit{Continuous vs. discrete models for the quantum harmonic oscillator and the hydrogen atom}, Phys. Letters A, \textbf{285}, no.~1  (2001), 119-126.}

\bibitem{rota70}{ Mullin  R. and Rota G-C - {\it On the foundations of combinatorial theory III. Theory of binomial enumeration}, In Harris, editor, Graph theory and its applications, pp. 167-213, Academic Press, 1970 }

\bibitem{HS01}{  Holland J. E. and  Sparling G. A. J., {\it Conformally invariant power of the ambient Dirac operators}, Arxive 2001.}





\end{thebibliography}
\end{document}